\documentclass[envcountsame,envcountsect]{svmult}
\usepackage{etex}
\usepackage{amsmath}
\usepackage{amscd,amssymb,amsbsy}
\usepackage{color}
\usepackage[all]{xy}
\usepackage[colorlinks,plainpages,backref,urlcolor=blue]{hyperref}
\usepackage{fancyvrb,bbm}
\usepackage[mathscr]{euscript}
\usepackage{tikz}
\usetikzlibrary{calc,intersections,arrows}
\usepackage{booktabs}

\spnewtheorem{rmkl}[theorem]{Remark}{\bfseries}{\upshape}
\spnewtheorem{exml}[theorem]{Example}{\bfseries}{\upshape}
\spnewtheorem{ques}[theorem]{Question}{\bfseries}{\upshape}
\spnewtheorem*{ack}{Acknowledgment}{\bfseries}{\upshape}

\newenvironment{exm}%
{\begin{exml}}%
{\hfill$\Diamond$\end{exml}}
\newenvironment{rmk}%
{\begin{rmkl}}%
{\hfill$\Diamond$\end{rmkl}}

\smartqed  

\newcommand{\abs}[1]{\left|#1\right|}
\newcommand{\set}[1]{\left\{#1\right\}}

\renewcommand{\P}{\mathbb{P}}

\newcommand{\C}{\mathbb{C}}

\renewcommand{\k}{\Bbbk}  
\newcommand{\M}{\mathsf{M}}    
\newcommand{\U}{\mathsf{U}}    
\newcommand{\R}{\mathcal{R}}   
\newcommand{\Bound}{\mathcal{S}} 
\newcommand{\MM}{\mathcal{M}}   
\newcommand{\oMM}{\overline{\MM}}   
\newcommand{\Co}{\mathcal{C}}   
\renewcommand{\L}{\mathcal{L}}   
\newcommand{\X}{\mathcal{X}}   

\DeclareMathOperator{\rank}{rank}      
\DeclareMathOperator{\spn}{span}      
\DeclareMathOperator{\chr}{char}      

\DeclareMathOperator{\ShExt}{\mathscr{E}\kern -0.5pt xt}

\DeclareMathOperator{\ShHom}{\mathscr{H}\kern -0.5pt om}
\DeclareMathOperator{\id}{id} 
\DeclareMathOperator{\im}{im} 
\DeclareMathOperator{\irr}{irr} 
\DeclareMathOperator{\ess}{ess} 
\DeclareMathOperator{\pc}{\|}    

\newcommand{\oV}{\overline{V}}  
\newcommand{\oA}{\overline{A}}  
\newcommand{\A}{{\mathcal{A}}}

\begin{document}

\title*{Higher resonance varieties}
\titlerunning{Higher resonance varieties of matroids}

\author{Graham Denham}
\institute{
Department of Mathematics,
University of Western Ontario, London, Ontario, Canada  N6A 5B7.
\email{gdenham@uwo.ca}\\
Partially supported by a grant from NSERC of Canada}


\maketitle
\setcounter{minitocdepth}{1}
\dominitoc

\abstract{
We present some new results about the resonance varieties of matroids
and hyperplane arrangements.  Though these have been the objects of
ongoing study, most work so far has focussed on cohomological
degree $1$.  We show that certain phenomena become apparent only by
considering all degrees at once.
}

\section{Introduction}
Resonance varieties are a cohomological invariant that first appeared
in the study of the cohomology of one-dimensional fundamental group
representations.  Though they were first considered for topological
spaces, they are algebraic in nature, and they may be defined for any
(differential) graded-commutative algebra~\cite{Su15,Su11}.

The resonance varieties associated with the complement of a complex
hyperplane arrangement are an interesting special case, and we mention 
in particular the surveys \cite{Fa10,Yu12} for their description of 
the history and motivation.  Here, the underlying topological
space is a complex, quasiprojective variety; however, it follows from the
Brieskorn-Orlik-Solomon Theorem~\cite{OS80} that the
resonance varieties depend only on combinatorial data coming from a
matroid.

With that in mind, it is tempting to ask for a formula or an efficient
algorithm that expresses the resonance varieties in terms of the matroid.
However, this seems to be a difficult problem.  In cohomological degree $1$,
Falk and Yuzvinsky~\cite{FY07} have given a characterization, building
on work of Libgober and Yuzvinsky~\cite{LY00} 
as well as Falk~\cite{Fa97}.  In degree greater than $1$, there has
been some progress (in particular \cite{Bu11,CDFV12}), but comparatively
little is known.  By way of contrast, Papadima and Suciu gave a 
closed formula for the resonance varieties of exterior Stanley-Reisner
rings (and right-angled Artin groups) in \cite{PS09}, building on work
of \cite{AAH99}.  Although the situation for matroids has some 
similarities, an analogous definitive result is so far out of reach.

Our main goal here is to develop some basic tools systematically.
We consider the  behaviour of resonance varieties of matroids and
arrangements under such constructions as weak maps, Gale duality, and 
the deletion/contraction constructions.  Combining these basic 
ingredients allows us to compute the resonance varieties of some 
graphic arrangements explicitly.  A broader range of phenomena appear
in moving from degree $1$ to $2$: as an example, we find a 
straightforward way to make hyperplane arrangements for which the
Milnor fibre $F$ has nontrivial monodromy eigenspaces in $H^2(F,\C)$.

A hyperplane arrangement is a matroid realization over $\C$.
Some results about resonance 
from the literature are known for all matroids, while others
depend on realizability.  Our approach is combinatorial, so matroids
(rather than arrangements) seem to be the appropriate objects for this
paper.  It was recently shown in \cite{PS14} that realizability 
imposes a non-trivial qualitative restriction on resonance varieties,
at least in positive characteristic.  This encourages us to keep
track of the role of realizability.

\subsection{Outline}
The paper is organized as follows.  We begin by recalling the definition
of the Orlik-Solomon algebra, viewed as a matroid invariant.  
We would like to make use of the naturality of the construction; however,
not all weak maps of matroids induce homomorphisms of Orlik-Solomon algebras.
In \S\ref{ss:cat} we impose a condition on weak maps to define a category 
$\MM$ of matroids on which the Orlik-Solomon construction is functorial.

In \S\ref{sec:res}, we define resonance varieties and review 
known results about some qualitative properties that distinguish the 
resonance varieties of Orlik-Solomon algebras from the general case.  For
example, for a matroid $\M$ of rank $\ell$, the resonance varieties 
are known to satisfy 
\[
\R^p(\M)\subseteq \R^{p+1}(\M)
\]
for $0\leq p<\ell$ \cite{EPY03}. 
At least in the realizable, characteristic-zero case, resonance varieties
are unions of linear subspaces.  In \S\ref{ss:bound}, we construct subspace
arrangements $\Bound^p(\M)$ that contain them, based on a result of 
Cohen, Dimca, and
Orlik~\cite{CDO03}.  We find that in some interesting cases this 
upper envelope is tight: i.e., $\R^p(\M)=\Bound^p(\M)$.

In \S\ref{sec:pieces}, we examine the effect of standard matroid
operations on resonance.  Some results are known, some folklore, and
others new.  We use these to compute some examples and find that
certain special properties of components of $\R^p(\M)$ for $p=1$ no
longer hold for $p\geq2$.

In the last section, we revisit the combinatorics of multinets
and singular subspaces in terms of maps of Orlik-Solomon algebras
as another attempt to characterize components of resonance varieties.
The results are inconclusive, although the last example strongly 
suggests that some interesting combinatorics remains to be uncovered.

\section{Background}
\subsection{Arrangements and matroids}\label{ss:arr and mat}
We refer to the books of Orlik and Terao~\cite{OTbook} and 
Oxley~\cite{Oxbook} for basic facts about hyperplane arrangements and 
matroids, respectively.  If $\M$ is a matroid on
the set $[n]:=\set{1,2,\ldots,n}$ and $\k$ is a field, let
$V=\k^n$, a vector space with a distinguished basis we will call 
$e_1,\ldots,e_n$.
The Orlik-Solomon algebra $A_{\k}(\M)$ is the quotient of an 
exterior algebra $E:=\Lambda(V)$ by an ideal $I=I(\M)$ generated by
homogeneous relations indexed by circuits in $\M$.  More explicitly, let
$\partial$ be the derivation on $E$ defined by $\partial(e_i)=1$ 
for all $1\leq i\leq n$.  Then $I$ is generated by 
\begin{equation}\label{eq:OSrelations}
\set{\partial(e_C)\colon \text{circuits $C\subseteq [n]$ of $\M$}},
\end{equation}
where $e_C:=\prod_{i\in C}e_i$, with indices taken in increasing order.
We will omit the $\M$ or $\k$ from the notation $A_\k(\M)$ where no
confusion arises.  If $i\in\M$ is a loop, then $C=\set{i}$ is a circuit
and $A(\M)=0$.

We regard 
hyperplane arrangements as linear representations of loop-free matroids.
For us, an arrangement $\A$ over a field $F$ is an ordered $n$-tuple of 
(nonzero) linear forms $(f_1,\ldots,f_n)$, where $f_i\in W^*$ for
$1\leq i\leq n$, and $W$ is a vector space over $F$.  We let $\M(\A)$
denote the matroid on $[n]$ whose dependent sets index the linear
dependencies of the $f_i$'s.
In particular, our arrangements are all central, and we explicitly allow
repeated hyperplanes.

If $\A$ is an arrangement, let 
$H_i=\ker(f_i)$, a hyperplane in $W$, for $1\leq i\leq n$.  
Let $M(\A)=W-\bigcup_{i=1}^n H_i$, and $U(\A)=\P W-\bigcup_{i=1}^n \P H_i$.
If $\A$ is an arrangement, we abbreviate $A_{\k}(\A):=A_{\k}(\M(\A))$.
If $F=\C$, the complement $M(\A)$ is a complex manifold, and 
the Brieskorn-Orlik-Solomon Theorem states that
$A_{\k}(\A)\cong H^*(M(\A),\k)$ as graded algebras.  

\subsection{Projectivization}
Suppose $A(\M)=E/I$ is the Orlik-Solomon algebra of a matroid $\M$.
Since $\partial^2=0$, it follows
$\partial(I)=0$, so $\partial$ induces a derivation on $A$ as well,
which we denote by $\partial_A$.  Let
\[
\oV=\ker(\partial|_{V})=
\Big\{v\in \k^n\colon \sum_{i=1}^nv_i=0\Big\},
\]
and let $\oA(\M)$ denote the subalgebra of $A$ generated by $\oV$.  
\begin{lemma}\label{lem:ker partial}
We have $\oA(\M)=\ker\partial_A$.
\end{lemma}
\begin{proof}
Clearly $\oA(\M)\subseteq\ker\partial_A$.  By 
\cite[Lemma~3.13]{OTbook}, the chain complex $(A,\partial)$ is exact,
so if $\partial_A(x)=0$, then $x=\partial_A(y)$ for some $y\in A$.  If
$e_I=e_{i_1}\cdots e_{i_k}\in E$, 
then $\partial(e_I)=(e_{i_1}-e_{i_k})(e_{i_2}-e_{i_k})\cdots 
(e_{i_{k-1}}-e_{i_k})$, which implies
$\partial(e_I)$ is in the subalgebra generated
by $\oV$.  By applying $\partial_A$ to a representative in $E$ for $y$, we
see $x\in \oA(\M)$.
\qed
\end{proof}
Together with the exactness of $(A,\partial_A)$, this gives a short exact
sequence of graded $\k$-modules,
\begin{equation}\label{eq:proj ses 1}
\begin{tikzpicture}[baseline=-3pt]
\node (a) at (0,0) {$0$};
\node (b) at (1.4,0) {$\oA(\M)$};
\node (c) at (3.0,0) {$A(\M)$};
\node (d) at (5.0,0) {$\oA(\M)[-1]$};
\node (e) at (6.7,0) {$0$.};
\draw[->] (a) to (b);
\draw[->] (b) to (c);
\draw[->] (c) to node[above] {$\scriptstyle \partial$} (d);
\draw[->] (d) to (e);
\end{tikzpicture}
\end{equation}

Of course, if $\M=\M(\A)$ where $\A$ is a complex arrangement, this
sequence has a well-known origin: the quotient map $M(\A)\to U(\A)$ makes
$M(\A)$ a split $\C^*$-bundle over $U(\A)$, so the induced algebra 
homomorphism $H^*(U(\A),\k)\to H^*(M(\A),\k)$ is injective.  In fact,
under the isomorphism $H^*(M(\A),\k)\cong A(\A)$, the Gysin map is 
identified with $\partial_A$, and $\oA(\A)\cong H^*(U(\A),\k)$: see 
\cite{CDFSSTY}
or \cite[\S6.1]{Dim92} for details.  With this in mind, we will call
$\oA(\M)$ the {\em projective Orlik-Solomon algebra} even if $\M$ does
not have a complex realization.
\subsection{A category of matroids}\label{ss:cat}
We would like to make use of maps of Orlik-Solomon algebras, so it will
be useful to have a functorial construction.  For this,
we recall the definition of a weak map of matroids from 
\cite[Ch.~9]{Wh86}. If $\M_1$ and $\M_2$ are matroids
on sets $S_1$ and $S_2$, respectively, we add a disjoint loop ``$0$'' to
$\M_i$ to make a matroid $\M_i^+$, for $i=1,2$.
\begin{definition}\label{def:weak}
A {\em weak map}
$f\colon \M_1\to \M_2$ is a map of sets $f\colon S_1\cup\set{0}\to
S_2\cup\set{0}$ with the following properties: $f(0)=0$, and
for all $I\subset S_1$, 
if $f|_I$ is injective and $f(I)$ is independent in $\M_2^+$, then 
$I$ is independent in $\M_1$.

We will say a weak map $f\colon \M_1\to\M_2$ is {\em non-degenerate} if
$f^{-1}(0)=\set{0}$ and {\em complete} if,
for every circuit $C$ of $\M_1$, we have $\abs{C\cap f^{-1}(0)}\neq 1$.
Clearly, non-degenerate weak maps are complete.
\end{definition}
\begin{proposition}\label{prop:composition}
If $f\colon \M_1\to\M_2$ and $g\colon \M_2\to\M_3$ are complete weak
maps, so is $g\circ f$.  If $f$ and $g$ are non-degenerate weak maps, 
so is $g\circ f$.
\end{proposition}
\begin{proof}
Weak maps are closed under composition, so the second assertion is
trivial, and it is enough to check that the composite of complete maps
is complete.  Suppose instead that there is a circuit $C$ in $\M_1$ and
a unique element $i\in C$ for which $g\circ f (i)=0$.  Since $f$ is 
complete, $f(i)\neq0$.  Then $f(i)$ is contained in a circuit 
$C'\subseteq f(C)$ of $\M_2$.  By assumption, $j=f(i)$ is the only element
of $C'$ which $g(j)=0$.  But $g$ is complete, a contradiction.
\qed
\end{proof}

In view of the previous result, matroids on finite sets form a category
with morphisms taken to be the complete weak maps, which we will denote
by $\MM$.  Let $\oMM$ denote the (wide) subcategory whose morphisms are
non-degenerate weak maps.
Our hypotheses (complete and non-degenerate) are designed
to make the Orlik-Solomon algebra and its projective version functorial,
respectively.

First, we note a categorical formality.
If $f\colon\M_1\to\M_2$ is a complete weak map and $f(S_1)\supseteq S_2$, then
it is easy to see $f$ is an epimorphism.  While we avoid trying to 
characterize epimorphisms, here is a convenient necessary condition.
\begin{lemma}\label{lem:epic}
If $f\colon \M_1\to\M_2$ is an epimorphism in $\MM$, then any element
$i\in S_2-f(S_1)$ is either a loop or it is parallel to an element of 
$f(S_1)$.
\end{lemma}
\begin{proof}
Suppose instead that some $i\in S_2-f(S_1)$ is neither a loop nor 
parallel to an element of $f(S_1)$.
Define two maps $g,h\colon \M_2\to \U_{1,1}$: let 
$g(j)=0$ for all $j\in S_2$ and $h(j)=0$ for all $j\neq i$, but $h(i)=1$.
Then $g$ and $h$ are complete weak maps, but $g\circ f=h\circ f$,
so $f$ is not an epimorphism.
\qed
\end{proof}
By the obvious 
action on the distinguished basis, a weak map $f\colon \M_1\to \M_2$
induces a linear map $V(\M_1)\to V(\M_2)$ which we will also denote by $f$.
\begin{lemma}\label{lem:weak map}  
Suppose $f\colon \M_1\to\M_2$ is a complete weak map of matroids.
Then $\Lambda(f)$ induces a well-defined map of 
$\k$-algebras, $A(\M_1)\to A(\M_2)$.  If $f$ is also non-degenerate, then
it restricts to a map of projective Orlik-Solomon algebras.
\end{lemma}
\begin{proof}
Write $A(\M_i)=E_i/I_i$ for $i=1,2$, and consider the homomorphism
$\Lambda(f)\colon E_1\to E_2$.  For any circuit $C$ in $\M_1$, we 
consider two cases.  If $0\not\in f(C)$, then $f(C)\cap S_2$ is
dependent, so the image of $\Lambda(f)(\partial(e_C))\in I_2$.  If
$0\in f(C)$, by completeness, $f(e_i)=f(e_j)$ for some distinct
$i,j\in C$, so $\Lambda(f)(\partial(e_C))=0$.  So $\Lambda(f)(I_1)\subseteq
I_2$, as required.

If, moreover, $f$ is non-degenerate, we see $\partial|_{V(\M_2)}\circ f=  
f\circ\partial|_{V(\M_1)}$, by evaluating on the distinguished basis.
Then $\im(f|_{\oV(\M_1)})\subseteq \oV(\M_2)$.  Since $\oA(\M_1)$ is 
generated in degree $1$, the image of $\Lambda(f)|_{\oA(\M_1)}$ lies
in $\oA(\M_2)$.
\qed
\end{proof}
If $f$ is a complete weak map, put 
$A(f)=\Lambda(f)\colon A(\M_1)\to A(\M_2)$.  If $f$ is also non-degenerate,
let $\oA(f)\colon \oA(\M_1)\to\oA(\M_2)$ denote the restriction.
\begin{theorem}\label{thm:functorial}
$A$ and $\oA$ are functors from $\MM$ and $\oMM$, respectively, 
to the category of graded-commutative
$\k$-algebras.  Moreover, $A$ preserves epimorphisms.
\end{theorem}
\begin{proof}
If $f$ is a morphism, $A(f)$ is determined by its action in degree $1$,
where $A$ obviously preserves composition, so $A$ is functorial.
Now suppose $f$ is an epimorphism.  
The algebra $A(\M_2)$ is spanned by monomials $e_I$, where $I$ is independent.  

If $I\subseteq\im(f)$, there exists a subset
$J\subseteq f^{-1}(I)$ with $\abs{J}=\abs{I}$.  If not, by Lemma~\ref{lem:epic},
we may replace some elements of $I$ with parallel elements to form a 
set $I'\subseteq \im(f)$ and find $J$ as above with $f(J)=I'$.  

Since $e_i=e_j$ for parallel elements in $A(\M_2)$, in both cases we have
$A(f)(e_J)=e_I$.  So $A(f)$ is a surjective ring homomorphism.
\qed
\end{proof}
\begin{exm}\label{ex:weak}
The map $f\colon \U_{2,3}\to \U_{2,2}$ given by $f(i)=i$ for $i=1,2$ and
$f(3)=0$ is a weak map.  Taking $C=[3]$, we see $f$ fails to be 
complete, and $\Lambda(f)$ fails to give a map of Orlik-Solomon 
algebras.

On the other hand, if $\M$ is the matroid of the graph
\begin{equation}\label{eq:pyramid}
G=\begin{tikzpicture}[baseline=(current bounding box.center),scale=0.6,
bend angle=25,vertex/.style={circle,draw,inner sep=1.3pt,fill=black}]

\node[vertex] (a) at (0,0) {};
\node[vertex] (b) at (2,0) {};
\node[vertex] (c) at (2,2) {};
\node[vertex] (d) at (0,2) {};
\node[vertex] (e) at (1,1) {};

\draw (a) to node[below] {1} (b);
\draw (b) to node[right] {2} (c);
\draw (c) to node[above] {3} (d);
\draw (d) to node[left] {4} (a);
\draw (a) to node[right] {5} (e);
\draw (b) to node[above] {6} (e);
\draw (c) to node[left] {7} (e);
\draw (d) to node[below] {8} (e);
\end{tikzpicture},
\qquad \M(G)\to \U_{3,4}:
\begin{tikzpicture}[baseline=(current bounding box.center),scale=0.6,
bend angle=25,vertex/.style={circle,draw,inner sep=1.3pt,fill=black}]

\node[vertex] (a) at (0,0) {};
\node[vertex] (b) at (2,0) {};
\node[vertex] (c) at (2,2) {};
\node[vertex] (d) at (0,2) {};
\node[vertex] (e) at (1,1) {};

\draw (a) to node[below] {1} (b);
\draw (b) to node[right] {2} (c);
\draw (c) to node[above] {3} (d);
\draw (d) to node[left] {4} (a);
\draw (a) to node[right] {4} (e);
\draw (b) to node[above] {1} (e);
\draw (c) to node[left] {2} (e);
\draw (d) to node[below] {3} (e);
\end{tikzpicture}
\end{equation}
then the map from the edges of $G$ to the set $[4]$ given by the edge labels
on the right
defines a non-degenerate weak map $\M(G)\to \U_{3,4}$, because $G$ contains no
three-cycles with distinct edge labels.  
\end{exm}

We will end this section with two easy but useful observations.
Suppose $\M$ is a matroid on $[n]$ and $\pi$ is a partition of $[n]$ into
$k$ parts.  Let $p_\pi\colon [n]\to[k]$ be the map that sends $i$ to $s$
whenever $i\in \pi_s$.
\begin{proposition}\label{prop:partition}
$p_\pi\colon \M\to \U_{2,k}$ is a morphism of $\oMM$
if and only if, whenever $i$ and $j$ are parallel in 
$\M$, $i$ and $j$ are in the same block of $\pi$.
\end{proposition}

\begin{definition}\label{def:simple}
If $\M$ is a matroid on $[n]$, we define an equivalence relation on $[n]$
by letting $i\sim j$ if and only if $\set{i,j}$ is dependent.  
The {\em simplification} of $\M$, denoted $\M_s$ is, by definition,
the induced matroid on the equivalence classes.  The
natural map $s\colon \M\to\M_s$ is a morphism of $\oMM$.
\end{definition}

The map $A(s)\colon A(\M)\to A(\M_s)$ is easily seen to be an 
isomorphism, where $I(\M)$ contains relations $e_i-e_j$ if $\set{i,j}$ is
dependent in $\M$.
\begin{rmk}\label{rem:simple}
Clearly the complement of a hyperplane arrangement is unaffected by the 
presence of
repeated hyperplanes, so for topological purposes there is no loss in assuming
that the underlying matroid is simple.  However, we will make some use of 
the fact that the isomorphism $A(s)$ is not an equality
in Theorem~\ref{thm:res_basics}.
\end{rmk}

\section{Resonance varieties}\label{sec:res}
Now suppose $E=\Lambda(V)$ is the exterior algebra over a (finite-dimensional)
$\k$-vector space $V$.  Suppose $A$ and $B$ are graded $E$-modules.
\subsection{Definitions}
For any $v\in V$, we have $v\cdot v=0$, so
there is a cochain complex of $\k$-modules, 
$(A,\cdot v)$, in which the differential is by right-multiplication by $v$.
This construction is natural, in the sense that if $f\colon A\to B$ is a 
graded $E$-module homomorphism, then for any $v\in V$, 
\begin{equation}\label{eq:natural}
f\colon (A,\cdot v)\to (B,\cdot f(v))
\end{equation}
is clearly a homomorphism of cochain complexes.

The resonance varieties of $A$ are defined for all integers $p,d\geq0$ to be
\[
\R^p_d(A)=\set{v\in V\colon \dim_\k H^p(A,\cdot v)\geq d},
\]
and we abbreviate $\R^p(A):=\R^p_1(A)$.  We note that our definition varies
slightly from the usual one (see, e.g., \cite{PS10}) in that we do not
assume either that $A$ itself is a $\k$-algebra or that $V=A^1$.  
The modules of greatest interest are, in fact, algebras $A=A(\M)=E/I$; 
however, we do allow $I$ to contain relations
of degree $1$, accommodating parallel elements in $\M$.
We suggest distinguishing the two parameters by referring to
$\R^p_d(A)$ with $p>1$ as ``higher'' resonance, versus ``deeper'' for $d>1$.  
Our focus here is on the former.

For any nonzero $v\in V$, clearly $v\in \R^p_d(A)$ if and only if 
$\lambda v\in\R^p_d(A)$ for any $\lambda\in\k^\times$, so they determine
projective subvarieties of $\P V$.

\subsection{Resonance of Orlik-Solomon algebras}
From now on, we restrict our attention to Orlik-Solomon algebras, 
and abbreviate: $\R^p(\M):= \R^p(A(\M))$ and $\R^p(\A):=\R^p(\M(\A))$ for
matroids $\M$ and arrangements $\A$, respectively.  If $G$ is a graph,
let $\M(G)$ denote its matroid, and $\R^p(G):=\R^p(\M(G))$.

First we mention some properties of resonance varieties that 
specific to Orlik-Solomon algebras.
One such feature is a nestedness property discovered by
Eisenbud, Popescu and Yuzvinsky~\cite{EPY03} and studied 
further in \cite{Bu11,DSY14}. 

\begin{theorem}\label{thm:propagation}
Let $\M$ be a matroid of rank $\ell$.  Then
\begin{equation}\label{eq:propagation}
\set{0}\subseteq \R^0(\M)\subseteq \R^1(\M)\subseteq\cdots\subseteq \R^\ell(\M)
\subseteq \overline{V}.
\end{equation}
\end{theorem}
\begin{proof}
The inclusions $\R^p(\M)\subseteq \R^{p+1}(\M)$ for $0\leq p\leq \ell-1$
were proven in \cite[Thm.~4.1(b)]{EPY03}.  The authors work with arrangements,
but their arguments apply to all matroids.  The inclusion 
$\R^p(\M)\subseteq\oV$ for all $p\geq0$ 
is due to Yuzvinsky, \cite[Prop.~2.1]{Yuz95}.
\qed
\end{proof}

Another is the following.  By contrast, this result depends on complex
geometry and a result due to Arapura~\cite{Ar97}: for a full explanation,
we refer to \cite{DPS09}.
\begin{theorem}\label{thm:linear}
Let $\A$ be a complex hyperplane arrangement, and $\k$ a field of characteristic
zero.
Then $\R^p_\k(\A)$ is a union of linear components, for $0\leq p\leq \rank(\A)$.
\end{theorem}

\begin{rmk}
In characteristic zero, then, resonance varieties of realizable matroids
are subspace arrangements.  Falk~\cite{Fa07} has shown that this is
not in general the case for $\chr\k\neq0$: see also \cite[Ex.~4.24]{Fa10}.  
So even for hyperplane arrangements, the resonance varieties depend on the
characteristic of the field (unlike the Orlik-Solomon algebra itself).
For a striking application of resonance in characteristic $3$, we refer
to Papadima and Suciu~\cite{PS14}.
\end{rmk}

A component $W$ of a resonance variety is called {\em essential} if 
$W\cap(\k^\times)^n\neq\emptyset$.
\begin{ques}
Assume $\chr\k=0$.  Then 
the components of $\R^1(\M)$ are linear for any matroid 
$\M$~\cite[Cor.\ 3.6]{LY00}.
Is $\R^p_{\k}(\M)$ a union of linear components for all matroids $\M$, 
for $p>1$?
\end{ques}
The next result, due to Libgober and Yuzvinsky~\cite{LY00}, is a qualitative
property of $\R^1(\M)$ which is both special to matroids and, we will see,
to degree $p=1$.
\begin{theorem}\label{thm:isotropic}
Assume $\chr\k=0$.  If
$\R^1(\M)$ contains a component $W$ of dimension $k>0$,
there is an injective homomorphism $\oA(\U_{2,k+1})\to \oA(\M)$.  Conversely,
the image of such a homomorphism in degree $1$ lies in $\R^1(\M)$.
\end{theorem}
\begin{proof}
Multiplication in $\oA^1(\U_{2,k+1})$ is zero, so this is just a reformulation
of the following result from \cite{LY00}: 
if $W$ is a component of $\R^1(\M)$, then for any $v,w\in W$, we have $vw=0$.
\qed
\end{proof}

Next, we see that questions of resonance can be reduced to the 
projective Orlik-Solomon algebra, via the short exact sequence 
\eqref{eq:proj ses 1}.

\begin{lemma}\label{lem:proj ses 2}
For any matroid $\M$ on $[n]$ and $v\in\oV$, there is a short exact sequence of 
cochain complexes
\begin{equation}\label{eq:proj ses}
\begin{tikzpicture}[baseline=-3pt]
\node (a) at (0,0) {$0$};
\node (b) at (2,0) {$(\oA(\M),v)$};
\node (c) at (4.5,0) {$(A(\M),v)$};
\node (d) at (7.4,0) {$(\oA(\M),v)[-1]$};
\node (e) at (9.4,0) {$0$.};
\draw[->] (a) to (b);
\draw[->] (b) to (c);
\draw[->] (c) to node[above] {$\scriptstyle \partial$} (d);
\draw[->] (d) to (e);
\end{tikzpicture}
\end{equation}
If $\chr\k\nmid n$, the sequence \eqref{eq:proj ses} is split.
\end{lemma}
\begin{proof}
The inclusion $\oA(\M)\to A(\M)$ makes $A(\M)$ an $\oA(\M)$-module.  
Using Lemma~\ref{lem:ker partial}, it is easily checked that $\partial$
is a $\oA(\M)$-module homomorphism.  With this, we see that the maps in
the sequence \eqref{eq:proj ses 1} commute with multiplication by $v$.

If $n$ is nonzero in $\k$, (left) multiplication by 
$\frac{1}{n}\sum_{i=1}^n e_i$ gives a right inverse to $\partial$, proving
the second assertion.
\qed
\end{proof}

With this, we see that the resonance of $A$ and $\oA$ 
differ only by a trivial factor.  
\begin{proposition}\label{prop:projective}
Let $\M$ be a matroid on $[n]$ of rank $\ell$.  If $\chr\k\nmid n$, then
$\R^p(\oA(\M))=\R^p(A(\M))$ for all $0\leq p\leq \ell-1$,
and $\R^{\ell}(A(\M))=\R^{\ell-1}(A(\M))$.  For all $d\geq0$, we also have
\begin{equation}\label{eq:cone-decone}
\R^p_d(A(\M))=\bigcup_{j\leq d}\R^p_j(\oA(\M)\cap \R^{p-1}_{d-j}(\oA(\M))
\end{equation}
\end{proposition}
\begin{proof}
The equality \eqref{eq:cone-decone} is a direct consequence of 
Lemma~\ref{lem:proj ses 2}.  In particular, for $d=1$, we have
\begin{equation}\label{eq:Rofdecone}
\R^p(A(\M))=\R^p(\oA(\M))\cup R^{p-1}(\oA(\M))
\end{equation}
for $0\leq p\leq\ell$.  We prove $\R^p(\oA(\M))=\R^p(A(\M))$ by induction.
The case $p=0$ follows from \eqref{eq:Rofdecone}.  The induction step
is obtained by combining \eqref{eq:Rofdecone} with 
Theorem~\ref{thm:propagation}.
\qed
\end{proof}

\subsection{Top and bottom}
The two ends of the resonance filtration \eqref{eq:propagation} have 
straightforward descriptions.  First, it will be convenient to have
some notation.

\begin{definition}\label{def:res subspaces}
If $\pi$ is a partition of $[n]$ with $k$ parts, let $P_\pi$ denote the
codimension-$k$ subspace of $\k^n$ given by equations
\[
\sum_{j\in \pi_i}x_j=0
\]
for $1\leq i\leq k$.  If $k=1$, 
we recover $\overline{V}=P_{\set{[n]}}$.  At the other extreme, if each
block of $\pi$ is a singleton, then $P_\pi=\set{0}$.

Dually, let $Q_\pi$ denote the $k$-dimensional subspace of $\k^n$ given as
the span of vectors
\[
\sum_{j\in\pi_i}e_j
\]
for $1\leq i\leq k$.  Clearly, $P_\pi$ and $Q_\pi$ are complementary
subspaces (with respect to the distinguished basis in $V$.)
\end{definition}

\begin{proposition}\label{prop:ends}
For any matroid $\M$ on $[n]$, let $\pi$ denote the partition of $[n]$
given by simplification (Definition~\ref{def:simple}.)  Then
$\R_1^0(\M)=P_\pi$.
\end{proposition}
\begin{proof}
The simplification map $s\colon \M\to \M_s$ (Definition~\ref{def:simple})
gives an isomorphism of complexes
\[
A(s)\colon (A(\M),v)\to (A(\M_s),s(v))
\]
for all $v\in V$.  For a simple matroid, clearly $\R^0(\M_s)=\set{0}$,
so $\R^0(\M)=s^{-1}(0)$, which is the subspace $P_\pi$.  
\qed
\end{proof}

At the other extreme, by the invariance of Euler characteristic, we
have
\[
\sum_{p=0}^{\ell-1}(-1)^p\dim_\k H^p(\oA(\M),v)=(-1)^{\ell-1}\beta(\M),
\]
for any $v\in \oV$, where $\beta(\M)$ denotes Crapo's beta invariant.
We recall that $\beta(\M)\neq 0$ if and only if $\M$ is connected.
If $\M=\M(\A)$ for an arrangement $\A$, it is usual to say
$\A$ is irreducible to mean $\M(\A)$ is connected.  

If $\M(\A)$ is connected, then for any $v\in \oV$, the nonzero Euler
characteristic implies that 
$H^p(\oA(\M),v)\neq 0$ for some $p$.  By Theorem~\ref{thm:propagation},
$H^{\ell-1}(\oA(\M),v)\neq 0$, which proves the following.
\begin{proposition}[\cite{Yuz95}]\label{prop:top}
If $\M$ is connected of rank $\ell$, then $\R^{\ell-1}(\oA(\M))=
\oV=P_{\set{[n]}}$.
\end{proposition}

\begin{exm}\label{ex:fat triangle}
Consider the matroid $\M(G)$, where
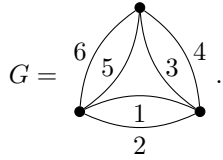
\[
G=\begin{tikzpicture}[baseline=(current bounding box.center),scale=0.8,
bend angle=25,vertex/.style={circle,draw,inner sep=1.3pt,fill=black}]

\node[vertex] (a) at (0,0)  {};
\node[vertex] (b) at (2,0) {};
\node[vertex] (c) at (1, 1.73) {};

\draw (a) to[bend right] node[below] {2} (b);
\draw (a) to[bend left] node[below] {1} (b);
\draw (b) to[bend right] node[right] {4} (c);
\draw (b) to[bend left] node[right] {3} (c);
\draw (c) to[bend right] node[left] {6} (a);
\draw (c) to[bend left] node[left] {5} (a);

\end{tikzpicture}.
\]
Using Proposition~\ref{prop:ends}, we see
$\R^0(G)=P_{12|34|56}=V(x_1+x_2,x_3+x_4,x_5+x_6)$, a $3$-dimensional 
subspace.  Since $\M(G)$ is connected of
rank $2$, $\R^1(G)=R^2(G)=\overline{V}=V(x_1+\cdots+x_6)$.
\end{exm}

\subsection{Upper bounds}\label{ss:bound}
Most of the results we present in this paper give lower bounds for the resonance
varieties: that is, conditions which imply nonvanishing cohomology.  Here,
we give two upper bounds.  The first uses a result of 
Schechtman and Varchenko~\cite{SV91} to 
give a weak but easily stated upper bound.  
\begin{theorem}\label{thm:upper bound}
If $\M$ is a matroid of rank $\ell$, then 
\[
\R^p(\M)\subseteq\bigcup_{X\in L^{\irr}_{\leq p+1}(\M)}
P_{\set{X,[n]-X}}
\]
for all $0\leq p\leq \ell$.
\end{theorem}
In other words, if $H^p(A(\M),v)\neq0$, then there exists an irreducible
flat $X$ of rank at most $p+1$ for which $v\in P_{\set{X,[n]-X}}$;
equivalently, for which $v\in \oV(\M)$ and $\sum_{i\in X}v_i=0$.
\begin{proof}
In \cite{SV91}, the authors construct a map of cochain complexes
\begin{equation}\label{eq:Smap}
\xymatrix{
(F(\M),d)\ar[r]^{S(v)} & (A(\M),v)
}
\end{equation}
for each $v\in V(\M)$.  The complex $(F(\M),d)$ is isomorphic to the
$\k$-dual of $(A(\M),\partial_A)$; in particular it is 
exact~\cite[Cor.~2.8]{SV91} and does not depend on $v$.  

Let $v(X):=\sum_{i\in X}v_i$.  The determinant
formula~\cite[Thm.~3.7]{SV91} expresses the determinant of $S(v)$ in terms
of powers of products of $v(X)$'s.  In particular, it follows that
$S^p(v)$ is an isomorphism if $v(X)\neq0$ for all $X\in L^{\irr}_{\leq p}(\M)$.
Suppose, then, that $v(X)\neq0$ for all irreducible $X$ of rank at most 
$p+1$.  This implies \eqref{eq:Smap} is an isomorphism up to degree $p+1$, so
$H^q(A(\M),v)=0$ for all $q\leq p$, from which the claim follows.
\qed
\end{proof}

Our second result gives a more refined bound for the largest (nontrivial)
resonance variety, based on the main result of Cohen, Dimca and 
Orlik~\cite{CDO03}.  Since their result applies to the cohomology
of local systems on a hyperplane complement, the proof of this bound
requires that the matroid have a complex realization.

\begin{definition}\label{def:covers}
Let us say that a subset of flats $\Co$ {\em covers} $\M$ if
there is a surjective function $f\colon[n]\to\Co$ for which $i\in f(i)$ for
all $1\leq i\leq n$.  For a given cover $\Co$, let 
$P_{\Co}=\bigcap_{X\in\Co}P_{\set{X,[n]-X}}$, a linear subspace of $\oV$.

For each $p\geq 0$, we define a subspace arrangement in $\oV$ using $\M$:
let
\[
\Bound^p(\M)=\bigcup_{\Co}P_{\Co},
\]
where the union is over all subsets $\Co\subseteq L^{\irr}_{\leq p+1}(\M)$ 
that cover $\M$.

Finally, say a cover $\Co$ of $\M$ is {\em essential} if $\abs{X}>1$ for all
$X\in\Co$, and let
\[
\Bound_{\ess}^p(\M)=\bigcup_{\Co}P_{\Co},
\]
where the union is over essential covers $\Co\subseteq L^{\irr}_{\leq p+1}(\M)$.
\end{definition}
We note that if $\M$ is irreducible of rank $\ell$, then $\Bound^0(\M)
\subseteq\cdots\subseteq\Bound^{\ell}(\M)=\oV$.
\begin{theorem}\label{thm:CDO bound}
If $\M$ is a complex-realizable matroid, then 
$\R^{p}(\M)\subseteq \Bound^{p}(\M)$,
for all $p\geq0$.
\end{theorem}
\begin{proof}
Suppose $\A$ is a complex arrangement of rank $\ell$ and $\M=\M(\A)$.
The main result of \cite{STV95} allows us to translate \cite[Thm.~1]{CDO03}
into the following statement about resonance.  That is, for $v\in \oV(\M)$, 
suppose that for some $i\in [n]$, whenever $i\in X\in L^{\irr}_{<\ell}(\M)$, 
we have $v(X)\neq0$.  Then $H^p(\oA(\M),v)=0$ for all $p\neq\ell-1$.

If $v\in\R^{\ell-2}(\M)$, then, for all
$i\in[n]$, there exists some $X$ for which $i\in X$ and $v\in 
P_{\set{X,[n]-X}}$, since $P_{\set{X,[n]-X}}=\set{v\in\oV\colon v(X)=0}$.
So 
\begin{align} \nonumber
\R^{\ell-2}(\M) & \subseteq 
\bigcap_{i=1}^n\bigcup_{
\substack{X\in L^{\irr}_{<\ell}\colon \\ i\in X}} 
P_{\set{X,[n]-X}},\\
& =
\bigcup_{\Co}\bigcap_{X\in\Co} P_{\set{X,[n]-X}}, \label{eq:CDO pf}\\
& = \Bound^{\ell-2}(M)\nonumber
\end{align}
since $\Co$ runs over all covers in $L^{\irr}_{\leq\ell-1}$.
To obtain the result for $p<\ell-1$, we consider the truncation $T_{p+1}\M$
of $\M$ to rank $p+1$.  This matroid is also realizable (by intersecting
$\A$ with a generic linear space of codimension $p+1$) and
$L_{\leq p}(\M)=L_{\leq p}(T_{p+1}\M)$, so it is enough to apply
\eqref{eq:CDO pf} to $T_{p+1}\M$.
\qed
\end{proof}
\begin{corollary}\label{cor:CDO cpt}
If $W$ is a component of $\R^{p}(\M)$ for a complex-realizable
matroid $\M$ and $p\geq0$, then there exists a cover $\Co\subseteq 
L^{\irr}_{\leq p+1}(\M)$ for which $W\subseteq P_{\Co}$.
\end{corollary}

\begin{corollary}\label{cor:ess bound}
If $\M$ is a complex-realizable matroid of rank $\ell$, then for all $p\geq0$,
\begin{equation}\label{eq:ess bound}
\R^{p}(\M)\cap(\k^\times)^n\subseteq \Bound_{\ess}^p(\M).
\end{equation}
\end{corollary}
\begin{proof}
If $X=\set{i}$, then $P_{\set{X,[n]-X}}$ is contained in the coordinate
hyperplane $x_i=0$, so its intersection with $(\k^\times)^n$ is empty.
\qed
\end{proof}
We will see examples in the next section in which the upper bound above
is sharp: that is, the containment \eqref{eq:ess bound} is an equality.
\begin{ques}
The bound given in Theorem~\ref{thm:CDO bound} is, of course, completely
combinatorial, although the result we use from \cite{CDO03} is topological
in nature.  Can the hypothesis that $\M$ is complex-realizable be dropped? 
\end{ques}
\section{Matroid operations and resonance}\label{sec:pieces}
In this section, we systematically 
examine the behaviour of resonance under several familiar matroid operations.  
In some cases, apparently nontrivial components 
of resonance varieties can be obtained from tautological components
belonging to other matroids.
\subsection{Naturality}
In some cases, resonance behaves well under the morphisms of $\MM$
from Section~\S\ref{ss:cat}.  Combining Theorem~\ref{thm:functorial} with
\eqref{eq:natural}, we see that if $f\colon \M_1\to\M_2$ is morphism, 
then for all $v\in V(\M_1)$, there is a map of complexes
\begin{equation}\label{eq:weak map}
A(f)\colon (A(\M_1),v)\to (A(\M_2),f(v)).
\end{equation}

Here is an important special case.  
\begin{proposition}\label{prop:submatroid}
If $X$ is a flat of a matroid $\M$, for any $v\in V(\M_X)$, the complex
$(A(\M_X),v)$ is a split subcomplex of $(A(\M),v)$.
\end{proposition}
\begin{proof}
The inclusion
$j\colon \M_X\to\M$ and projection $\phi\colon \M\to\M_X$
given by
\[
\phi(i)=\begin{cases}
i&\text{if $i\in X$;}\\
0&\text{otherwise}
\end{cases}
\]
are both complete
weak maps.  Since $\phi\circ j=\id$, the result follows by naturality.
\qed
\end{proof}

\begin{lemma}\label{lem:slice}
Suppose that $f\colon \M_1\to\M_2$ is a morphism induced by the identity map
on underlying sets.  
If the posets $L_{\leq k}(\M_1)\cong L_{\leq k}(\M_2)$ 
for some integer $k\geq 1$, then
$\R^p(\M_1)\cong \R^p(\M_2)$ for $0\leq p\leq k-1$.
\end{lemma}
\begin{proof}
Consider the map $A(f)\colon (A(\M_1),v)\to (A(\M_2),v)$, for some 
$v\in V(\M_1)$.  The hypotheses imply this is an isomorphism in
degrees $\leq k$, 
so $H^p(A(\M_1),v)\cong H^p(A(\M_2),v)$ for all $0\leq p\leq k-1$.
\qed
\end{proof}
We say that a matroid $\M$ is $k$-generic for some $k\geq0$ if
$\M$ has no circuits of size $k$ or smaller.  

\begin{proposition}\label{prop:kgeneric}
Suppose $\M$ is $k$-generic.  Then
$\R^p(\M)=\set{0}$ for all $i$, $0\leq p\leq k-1$.
\end{proposition}
\begin{proof}
Let $F_n$ denote the free matroid on $[n]$ (the underlying matroid of the
Boolean arrangement.)  Clearly
$A(F_n)=E$, the exterior algebra, and it is straightforward to check
that $\R^p(E)=\set{0}$ for $0\leq p\leq n$.  The claim then follows from
Lemma~\ref{lem:slice}.
\qed
\end{proof}
\begin{exm}[Uniform matroids]\label{ex:uniform}
Consider the uniform matroid $\U_{\ell,n}$ of rank $\ell$ on $[n]$.
Assume $n>\ell$.  Then $\U_{\ell,n}$ is $k$-generic for all $k<\ell$,
so $\R^p(\U_{\ell,n})=\set{0}$ for $0\leq p<\ell-1$, and 
$\R^{\ell-1}(\U_{\ell,n})=\oV$ because $\U_{\ell,n}$ is connected.
\end{exm}

\subsection{Sums, submatroids and duals}
Let us single out three particularly well-behaved constructions.
For convenience, if $X\in L(\M)$, we abuse notation and identify
$V(\M_X)$ with its image as a coordinate subspace in $V(\M)$, supported
in the coordinates indexed by $X$.
\begin{theorem}\label{thm:res_basics}
Suppose $\M$ is a matroid of rank $\ell$ on $[n]$.  Let $X$ be any flat
of $\M$, and $\M^\perp$ the dual matroid.  Then:
\[
\begin{array}{lll} 
\toprule
& \text{Construction~~~} & \text{Resonance}\\ \midrule
(1) & \M=\M_1\oplus \M_2 & 
\R^k(\M)=\bigcup_{p+q=k}\R^p(\M_1)\times \R^q(\M_2) \\
(2) & \text{submatroids } & 
\R^p(\M_X)\subseteq  \R^p(\M)\cap V(\M_X)\text{~for all $p\geq 0$}\\
(3) & \text{duality} & \R^{\ell-p}(\M)\cap (\k^\times)^n = 
\R^{n-\ell-p}(\M^\perp)\cap (\k^\times)^n, 
\text{~for $p\geq 0$}\\ \bottomrule
\end{array}
\]
\end{theorem}
\begin{proof}
$(1)$ is due to Papadima and Suciu~\cite[Prop.~13.1]{PS10},
since if $\M=\M_1\oplus \M_2$, then $A(\M)=A(\M_1)\otimes_\k A(\M_2)$.
Claim $(2)$ follows directly from Proposition~\ref{prop:submatroid}.
Claim $(3)$ appears as Theorem~27 in \cite{De01}.  
\qed
\end{proof}
\subsection{Local components}
If $X$ is a flat of an arrangement $\M$ on $[n]$, 
we define a partition $\pi(X)$: let $X=X_1\vee X_2\cdots \vee X_r$
be the decomposition of $X$ into connected components.  Let $\pi(X)$ be
the partition given by $i\sim j\Leftrightarrow \set{i,j}\subseteq X_k$, 
for all $i,j,k$, and $i\sim i$ if $i\not\in X$.

\begin{proposition}\label{prop:local}
For any matroid $\M$, then for any flat $X\in L_q(\M)$, 
there is an equality $P_{\pi(X)}= V_X\cap R^p(\M)$ for all
$p$ with $q-r\leq p\leq\ell$, where $r=\abs{\pi(X)}$.
\end{proposition}
\begin{proof}
Since 
$
\M_X=\bigoplus_{i=1}^r \M_{X_i}
$,
in view of the product formula of Theorem~\ref{thm:res_basics}(1), it
is enough
to prove the claim when $X$ is connected.  In this case, $P_{\pi(X)}=
\overline{V}_X$.  By
Proposition~\ref{prop:top}, $\R^{q-1}(\M_X)=\overline{V}_X$, which is 
contained in $\R^{q-1}(\M)\cap V_X$ by Theorem~\ref{thm:res_basics}(2).
By propagation (Theorem~\ref{thm:propagation}), 
\begin{equation}\label{eq:local}
\overline{V}_X \subseteq\R^p(\M)\cap V_X\text{~for $q-1\leq p\leq \ell$.}
\end{equation}
On the other hand, $\R^{p}(\M)\subseteq \overline{V}$ and 
$\overline{V}\cap V_X=\overline{V}_X$, so \eqref{eq:local} is an equality
for all $p$ with $q-1\leq p\leq \ell$.
\qed\end{proof}
\begin{exm}\label{ex:Xgraph}
Let us consider the dual to Example~\ref{ex:fat triangle}.  
That is, $\M^\perp=\M(G^\perp)$, where
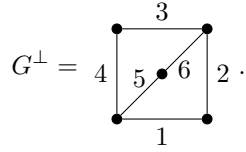
\[
G^\perp=
\begin{tikzpicture}[baseline=(current bounding box.center),scale=0.6,
bend angle=25,vertex/.style={circle,draw,inner sep=1.3pt,fill=black}]

\node[vertex] (a) at (0,0) {};
\node[vertex] (b) at (2,0) {};
\node[vertex] (c) at (2,2) {};
\node[vertex] (d) at (0,2) {};
\node[vertex] (e) at (1,1) {};

\draw (a) to node[below] {1} (b);
\draw (b) to node[right] {2} (c);
\draw (c) to node[above] {3} (d);
\draw (d) to node[left] {4} (a);
\draw (a) to node[above] {5} (e);
\draw (e) to node[below] {6} (c);

\end{tikzpicture}.
\]
By Proposition~\ref{prop:kgeneric}, 
we see $\R^0(\M^\perp)=\R^1(\M^\perp)=\set{0}$.
By Proposition~\ref{prop:local},
 each of the circuits of size $4$ contribute
local components, so each of the $3$-dimensional linear spaces
$P_{1|2|3456}$, $P_{3|4|1256}$, $P_{5|6|1234}$
are contained in $\R^2(\M^\perp)$.  

By Theorem~\ref{thm:res_basics}(3), we see 
$\R^2(\M^\perp)$ also contains the component
$P_{12|34|56}$.  With the help of Macaulay2~\cite{M2}, we find that
the upper bound from Theorem~\ref{thm:CDO bound} is sharp, so
\[
\R^2(\M^\perp)=\Bound^2(\M^\perp)=
P_{1|2|3456}\cup P_{3|4|1256}\cup P_{5|6|1234}\cup P_{12|34|56}.
\]
Since $\M^\perp$ is irreducible of rank $4$, we have 
$\R^3(\M^\perp)=\R^4(\M^\perp)=P_{123456}$, by Proposition~\ref{prop:ends}.
\end{exm}
\begin{rmk}
As an aside, let's recall an application to the cohomology of the Milnor
fibration of a hyperplane arrangement, a full treatment of which can
be found in the paper of Papadima and Suciu~\cite{PS14}.  
The matroid $\M^\perp$ is realized
by the arrangement defined by 
\[
Q=(x_1-x_2)(x_2-x_3)(x_3-x_4)(x_1-x_4)(x_1-x_5)(x_3-x_5).
\]
Each part of $\pi=12|34|56$ has
even length, so the vector $d=(1,1,1,1,1,1)\in\R^2_{\k}(\M^\perp)$ if 
$\chr\k=2$.  From an application of the tangent cone formula (see, for example,
\cite{DS12}), it follows $-1$ is a monodromy eigenvalue of 
$H^2(F,\C)$, where $F=Q^{-1}(1)$, the Milnor fibre of the arrangement.  
Our point here is that
an apparently nontrivial cohomological phenomenon can result from 
matroid operations applied to a rather trivial example 
(Example~\ref{ex:fat triangle}).
\end{rmk}
\subsection{Deletion-contraction}
In the case of complex hyperplane arrangements, Cohen~\cite{Co02} showed that
there is a long exact sequence relating cohomology of a local system on
the complement of an arrangement, its deletion, and its restriction.  We 
will use the combinatorial analogue here in order to examine the 
behaviour of resonance varieties under deletion and contraction.

If $\M$ is a matroid and $i_0$ is an element which is not a loop, let
$\M':=\M-\set{i_0}$ and $\M'':=\M/i_0$ denote
the deletion and contraction matroids,
respectively.  From \cite[Thm.~3.65]{OTbook}, there is a short exact sequence
of $\k$-modules
\begin{equation}\label{eq:delres1}
\begin{tikzpicture}[baseline=-3pt]
\node (a) at (0,0) {$0$};
\node (b) at (1.4,0) {$A(\M')$};
\node (c) at (3,0) {$A(\M)$};
\node (d) at (5.0,0) {$A(\M'')[-1]$};
\node (e) at (6.7,0) {$0$};
\draw[->] (a) to (b);
\draw[->] (b) to node[above] {$\scriptstyle j$} (c);
\draw[->] (c) to node[above] {$\scriptstyle \phi$} (d);
\draw[->] (d) to (e);
\end{tikzpicture}
\end{equation}
where the map $\phi$ is defined on monomials by 
\[
\phi(e_{i_0}e_S)=e_S,
\]
and $\phi(e_S)=0$ if $i_0\not\in S$.  The inclusion $j$ is a ring homomorphism,
induced by the morphism $\M'\hookrightarrow\M$, so \eqref{eq:delres1}
is a sequence of $A(\M')$-modules.

For any $v\in\overline{V}(\M')$, then, we obtain a short exact sequence
of complexes,
\begin{equation}\label{eq:delres2}
\begin{tikzpicture}[baseline=-3pt]
\node (a) at (0,0) {$0$};
\node (b) at (1.7,0) {$(A(\M'),v)$};
\node (c) at (4.2,0) {$(A(\M),v)$};
\node (d) at (7.2,0) {$(A(\M''),v)[-1]$};
\node (e) at (9.2,0) {$0$};
\draw[->] (a) to (b);
\draw[->] (b) to node[above] {$\scriptstyle j$} (c);
\draw[->] (c) to node[above] {$\scriptstyle \phi$} (d);
\draw[->] (d) to (e);
\end{tikzpicture}
\end{equation}

where we identify $\overline{V}(\M')$ and $\overline{V}(\M'')$ 
with the coordinate hyperplane
$H_{i_0}:=P_{i_0|1,\ldots,\hat{i_0},\ldots,n}\subseteq \overline{V}(\M)$.

\begin{proposition}\label{prop:delres fine}
Let $(\M, \M',\M'')$ be a deletion-contraction triple.  Then, for all 
integers $p\geq0$ and $k\geq j\geq 0$, we have the following inclusions:
\begin{align}
H_{i_0}\cap \R_k^p(\M)-\R_{j+1}^p(\M') & \subseteq \R_{k-j}^{p-1}(\M'');
\label{eq:propdelres1}\\
\R_k^{p-1}(\M'')-\R_{j+1}^p(\M) & \subseteq \R_{k-j}^{p+1}(\M');
\label{eq:propdelres2}\\
\R_k^p(\M')-\R_{j+1}^{p-2}(\M'') & \subseteq \R_{k-j}^p(\M).
\label{eq:propdelres3}
\end{align}
\end{proposition}
\begin{proof}
Given $v\in H_{i_0}$, consider the long exact sequence in cohomology of
\eqref{eq:delres2}.  By exactness, we have
\[
\dim_\k H^p(A(\M),v)\leq \dim_\k H^p(A(\M'),v)+\dim_\k H^{p-1}(A(\M''),v),
\]
from which the first containment easily follows.  The remaining two are obtained
in the same way.
\qed\end{proof}
We note that, by putting $j=0$ and $k=1$, we can erase the subscripts
that keep track of depth.
\begin{corollary}\label{cor:kgeneric delres}
Suppose $\M$ is $k$-generic.  Then for any element $i_0$, 
\[
\R^p(\M-\set{i_0})\subseteq\R^p(\M),
\]
for all $0\leq p\leq k$.
\end{corollary}
\begin{proof}
If $\M$ is $k$-generic, the contraction $\M/i_0$ is $(k-1)$-generic.  The
result then follows by Propositions~\ref{prop:delres fine}\eqref{eq:propdelres3}
and \ref{prop:kgeneric}.
\qed\end{proof}
Since all simple matroids are $1$-generic, the situation for
$p=1$ is particularly nice.  The next example shows 
that, in general, $\R^p(\M')\not\subseteq\R^p(\M)$ for $p>1$.
\begin{exm}\label{ex:delres}
Let $\M=\M(G_1)$ be the matroid of the graph shown in Figure~\ref{fig:delres}.
Deleting the horizontal edge gives the graphic matroid of 
Example~\ref{ex:Xgraph}, which we denote for the moment by $\M'$.  
We saw $W:=P_{12|34|56|7}\subseteq\R^2(\M')$: let us parameterize
\[
W=\set{(-a,a,b,-b,-c,c,0)\colon a,b,c\in\k}.
\]
If $v\in W$, then $v\not\in \R^0(\M'')$ provided either $a+b\neq 0$ or $c\neq0$.
From Proposition~\ref{prop:delres fine}\eqref{eq:propdelres3}, we
see $W\subseteq\R^2(\M)$.  We can repeat this twice to find
$W\subseteq\R^2(G_2)$ as well, where $G_2$ is the graph 
obtained from $K_5$ by deleting an edge.

However, the same argument does not allow us to conclude $W\subseteq\R^2(K_5)$:
if we contract the bottom edge, we see
$W$ is now contained in $\R^0$ of the contraction
(Proposition~\ref{prop:ends}), so Proposition~\ref{prop:delres fine} does
not apply.  Indeed, it turns out
that $W\not\subseteq\R^2(K_5)$.  Since $W\subseteq \Bound^2(\M(K_5))$, 
though, our upper bound is of no use here, and we are forced to verify this 
with a direct calculation.  
\end{exm}
\begin{figure}
\begin{center}
\begin{tikzpicture}[baseline=(current bounding box.center),scale=0.9,
vertex/.style={circle,draw,inner sep=1.3pt,fill=black}]
\node[vertex] (1) at (90:4em) {};
\node[vertex] (2) at (162:4em) {};
\node[vertex] (3) at (234:4em) {};
\node[vertex] (4) at (306:4em) {};
\node[vertex] (5) at (378:4em) {};
\node at (270:5.5em) {$\subseteq\R^2(G_1)$};
\draw[blue,thick] (3) -- node[left] {$-a$} (2) -- node[above, near start] 
{$a$} (4);
\draw[red,thick] (3) -- node[above,near end] {$-b$} (5) -- node[right] {$b$} 
(4);
\draw[green!50!black,thick] (3) -- node[left,very near end] {$-c$} (1) -- 
node[right,very near start] {$c$} (4);
\draw[very thin] (2) -- node[below] {$0$} (5);
\end{tikzpicture}
\qquad
\begin{tikzpicture}[baseline=(current bounding box.center),scale=0.9,
vertex/.style={circle,draw,inner sep=1.3pt,fill=black}]
\node[vertex] (1) at (90:4em) {};
\node[vertex] (2) at (162:4em) {};
\node[vertex] (3) at (234:4em) {};
\node[vertex] (4) at (306:4em) {};
\node[vertex] (5) at (378:4em) {};
\node at (270:5.5em) {$\subseteq\R^2(G_2)$};
\draw[blue,thick] (3) -- node[left] {$-a$} (2) -- node[above,near start] {$a$}
 (4);
\draw[red,thick] (3) -- node[above,near end] {$-b$} (5) -- node[right] {$b$} (4);
\draw[green!50!black,thick] (3) -- node[left,near end] {$-c$} (1) -- 
node[right,near start] {$c$} (4);
\draw[very thin] (2) -- node[below] {$0$} (5) -- node[above] {$0$} (1) --
node[above] {$0$} (2);
\end{tikzpicture}
\qquad
\begin{tikzpicture}[baseline=(current bounding box.center),scale=0.9,
vertex/.style={circle,draw,inner sep=1.3pt,fill=black}]
\node[vertex] (1) at (90:4em) {};
\node[vertex] (2) at (162:4em) {};
\node[vertex] (3) at (234:4em) {};
\node[vertex] (4) at (306:4em) {};
\node[vertex] (5) at (378:4em) {};
\node at (270:5.5em) {$\not\subseteq\R^2(K_5)$};
\draw[blue,thick] (3) -- node[left] {$-a$} (2) -- node[above,near start] {$a$}
 (4);
\draw[red,thick] (3) -- node[above,near end] {$-b$} (5) -- node[right] {$b$} (4);
\draw[green!50!black,thick] (3) -- node[left,near end] {$-c$} (1) -- 
node[right,near start] {$c$} (4);
\draw[very thin] (2) -- node[below] {$0$} (5) -- node[above] {$0$} (1) --
node[above] {$0$} (2);
\draw[very thin] (3) -- node[below] {$0$} (4);
\end{tikzpicture}
\end{center}
\caption{Adding edges may not preserve resonance (Example~\ref{ex:delres})}\label{fig:delres}
\end{figure}
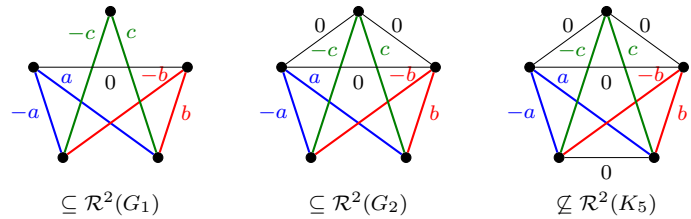

Along the same lines, we see also that a component of $\R^2(\M)$ contained
in a coordinate hyperplane need not be a component of the deletion $\R^2(\M')$.

\begin{exm}\label{ex:pyramid1}
Consider the matroid of the graph $G$ from Example~\ref{ex:weak},
ordered as shown in \eqref{eq:pyramid}.
With the help of Macaulay~2 and Theorem~\ref{thm:CDO bound}, we see
\[
W:=\set{(-a-b,b,0,a,b,a,-a,-b)\colon a,b\in\k}
\]
is a component of
$\R^2(G)$.  Let $G'$ be the graph obtained by deleting the edge with
label $0$ (see Figure~\ref{fig:delres2}).  One argues that $W\not\subseteq
\R^2(G')$ using Theorem~\ref{thm:CDO bound} as follows.  
If $v\in W$ and $a,b\neq0$, we choose an edge incident to a degree
vertex, and check that the only irreducible flats $X$ that contain it
have $v(X)\neq0$.  It follows that $v\not\in \Bound^2(G')$.
\begin{figure}
\begin{center}
\begin{tikzpicture}[baseline=(X.base),scale=0.75,
bend angle=25,vertex/.style={circle,draw,inner sep=1.3pt,fill=black}]

\node[vertex] (a) at (0,0) {};
\node[vertex] (b) at (2,0) {};
\node[vertex] (c) at (2,2) {};
\node[vertex] (d) at (0,2) {};
\node[vertex] (e) at (1,1) {};
\node (X) at (1,-1) {$W\not\subseteq\R^2(G')$};

\draw[green!50!black,thick] (a) to node[below] {$-a-b$} (b);
\draw[red,thick] (b) to node[right] {$b$} (c);
\draw[blue,thick] (d) to node[left] {$a$} (a);
\draw[red,thick] (a) to node[right] {$b$} (e);
\draw[blue,thick] (b) to node[above] {$a$} (e);
\draw[green!50!black,thick] (c) to node[left] {$-a$} (e);
\draw[green!50!black,thick] (d) to node[below] {$-b$} (e);
\end{tikzpicture}
\qquad
\begin{tikzpicture}[baseline=(X.base),scale=0.75,
bend angle=25,vertex/.style={circle,draw,inner sep=1.3pt,fill=black}]

\node[vertex] (a) at (0,0) {};
\node[vertex] (b) at (2,0) {};
\node[vertex] (c) at (2,2) {};
\node[vertex] (d) at (0,2) {};
\node[vertex] (e) at (1,1) {};
\node (X) at (1,-1) {$W\subseteq\R^2(G)$};

\draw[green!50!black,thick] (a) to node[below] {$-a-b$} (b);
\draw[red,thick] (b) to node[right] {$b$} (c);
\draw[very thin] (c) to node[above] {$0$} (d);
\draw[blue,thick] (d) to node[left] {$a$} (a);
\draw[red,thick] (a) to node[right] {$b$} (e);
\draw[blue,thick] (b) to node[above] {$a$} (e);
\draw[green!50!black,thick] (c) to node[left] {$-a$} (e);
\draw[green!50!black,thick] (d) to node[below] {$-b$} (e);
\end{tikzpicture}
\qquad
\begin{tikzpicture}[baseline=(X.base),scale=0.75,
bend angle=25,vertex/.style={circle,draw,inner sep=1.3pt,fill=black}]
\node[vertex] (1) at (90:4em) {};
\node[vertex] (2) at (210:4em) {};
\node[vertex] (3) at (330:4em) {};
\node[vertex] (4) at (0,0) {};
\node (X) at (0,-1.7) {$W\subseteq\R^1(G'')$};
\draw[green!50!black,thick] (4) to node[near start] (labl) {} (1);
\node[green!50!black] (text) at (40:5.5em) {$-a-b$};
\path[->] (text) edge  (labl);
\draw[green!50!black,thick] (2) to node[below] {$-a-b$} (3);
\draw[red,thick] (1) to node[right] {$b$} (3);
\draw[red,thick] (2) to node[right] {$b$} (4);
\draw[blue,thick] (1) to node[left] {$a$} (2);
\draw[blue,thick] (3) to node[above] {$a$} (4);

\end{tikzpicture}
\end{center}
\caption{Deletion may not preserve resonance (Example~\ref{ex:pyramid1})}
\label{fig:delres2}
\end{figure}
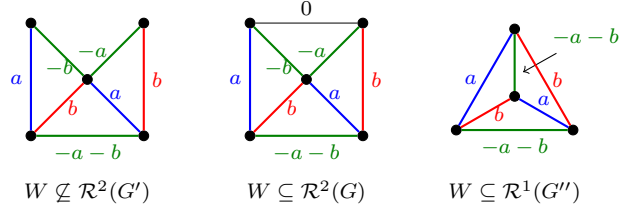
\end{exm}

\subsection{Parallel connections}\label{subsec:pc}
Here, we consider another matroid construction through which resonance varieties
can be traced.  The underlying data is the following.  Suppose $\M_1$ and
$\M_2$ are matroids on ground sets $S_1$ and $S_2$, respectively, and
$X=S_1\cap S_2$ is a modular flat of $\M_1$ and $(\M_1)_X=(\M_2)_X$.  The
(generalized) 
parallel connection $\M_1\pc_X\M_2$ is the matroid on $S_1\cup S_2$ obtained
from $\M_1\oplus \M_2$ by identifying the common copy of $X$ -- see 
\cite{Oxbook} for details.  Let $\M_{12}=(\M_1)_X=(\M_2)_X$.

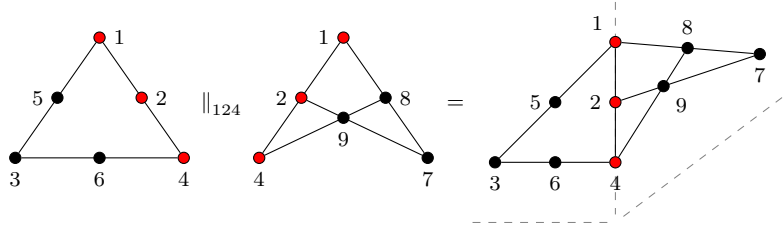
\begin{figure}
\[
\begin{tikzpicture}[baseline=(current bounding box.center),scale=0.8,
vertex/.style={circle,draw,inner sep=1.5pt,fill=black}]
\node[vertex] (3) at (-1.4,-1) [label=below:$3$] {};
\node[vertex,fill=red] (1) at (0,1) [label=right:$1$] {};
\node[vertex,fill=red] (4) at (1.4,-1) [label=below:$4$] {};
\draw (1) -- (4);
\node[vertex,fill=red] (2) at ($(1)!0.5!(4)$) [label=right:$2$] {};
\node[vertex] (6) at ($(3)!0.5!(4)$) [label=below:$6$] {};
\node[vertex] (5) at ($(1)!0.5!(3)$) [label=left:$5$] {};
\draw (1) -- (3);
\draw (3) -- (4);
\end{tikzpicture}
\pc_{124}
\begin{tikzpicture}[baseline=(current bounding box.center),scale=0.8,
vertex/.style={circle,draw,inner sep=1.5pt,fill=black}]
\node[vertex,fill=red] (4) at (-1.4,-1) [label=below:$4$] {};
\node[vertex,fill=red] (1) at (0,1) [label=left:$1$] {};
\node[vertex] (7) at (1.4,-1) [label=below:$7$] {};
\draw (1) -- (4);
\node[vertex,fill=red] (2) at ($(1)!0.5!(4)$) [label=left:$2$] {};
\node[vertex] (8) at ($(1)!0.5!(7)$) [label=right:$8$] {};
\draw (1) -- (7);
\draw[name path=a] (2) -- (7);
\draw[name path=b] (4) -- (8);
\path[name intersections={of=a and b,by=9}];
\node[vertex] at (9) [label=below:$9$] {};
\end{tikzpicture}
=
\begin{tikzpicture}[baseline=(current bounding box.center),scale=0.8,
vertex/.style={circle,draw,inner sep=1.5pt,fill=black}]
\node (O) at (0,-2) {};
\node (Y) at (0,1.8) {};
\draw[style=help lines,dashed] (Y) -- (O);
\node[vertex,fill=red] (1) at (0,1) [label=above left:$1$] {};
\node[vertex,fill=red] (4) at (0,-1) [label=below:$4$] {};
\draw (1) -- (4);  
\node[vertex,fill=red] (2) at ($(1)!0.5!(4)$) [label=left:$2$] {};
\node[vertex] (3) at (-2,-1) [label=below:$3$] {};
\node[vertex] (5) at ($(1)!0.5!(3)$) [label=left:$5$] {};
\node[vertex] (6) at ($(3)!0.5!(4)$) [label=below:$6$] {};
\draw (1) -- (3) -- (4);
\node[vertex] (7) at (2.4,0.8) [label=below:$7$] {};
\node[vertex] (8) at ($(1)!0.5!(7)$) [label=above:$8$] {};
\draw (1) -- (7);
\draw[name path=a] (2) -- (7);
\draw[name path=b] (4) -- (8);
\path[name intersections={of=a and b,by=9}];
\node[vertex] at (9) [label=below right:$9$] {};
%
%
\node (X) at ($(O)+(3)-(4)$) {};
\node (Z) at ($(O)+(7)-(4)$) {};
\draw[style=help lines,dashed] (O) -- ($(O)!1.2!(Z)$);
\draw[style=help lines,dashed] (O) -- ($(O)!1.2!(X)$);
\end{tikzpicture}
\]
\caption{Parallel connection}\label{fig:pc}
\end{figure}

The Orlik-Solomon complex of the parallel connection can be
described as follows.  We begin with the degree-$1$ part.
Let $i_j\colon \M_{12}\to \M_{j}$ denote the inclusions, for $j=1,2$.
The identification map $\phi\colon S_1\sqcup S_2\to S_1\cup S_2$.
These are all morphisms of $\MM$, and they fit in a short exact sequence:
\begin{equation*}
\begin{tikzpicture}[baseline=-3pt]
\node (a) at (0,0) {$0$};
\node (b) at (1.5,0) {$V(\M_{12})$};
\node (c) at (4.5,0) {$V(\M_1)\oplus V(\M_2)$};
\node (d) at (7.5,0) {$V(\M_1\pc_X \M_2)$};
\node (e) at (9.5,0) {$0$};
\draw[->] (a) to (b);
\draw[->] (b) to node[above] {$\scriptstyle i_1\oplus-i_2$} (c);
\draw[->] (c) to node[above] {$\scriptstyle \phi$} (d);
\draw[->] (d) to (e);
\end{tikzpicture}
\end{equation*}
By restricting, we also obtain:
\begin{equation}\label{eq:pc_Vs}
\begin{tikzpicture}[baseline=-2pt]
\node (a) at (0,0) {$0$};
\node (b) at (1.5,0) {$\oV(\M_{12})$};
\node (c) at (4.5,0) {$\oV(\M_1)\oplus V(\M_2)$};
\node (d) at (7.5,0) {$\oV(\M_1\pc_X \M_2)$};
\node (e) at (9.5,0) {$0$};
\draw[->] (a) to (b);
\draw[->] (b) to (c);
\draw[->] (c) to node[above] {$\scriptstyle \phi$} (d);
\draw[->] (d) to (e);
\end{tikzpicture}
\end{equation}

Now let $P$ denote the pushout of graded $\k$-algebras:
\begin{equation}\label{eq:pc pushout}
\begin{gathered}
\xymatrix{
\ar @{} [dr] | {\big\ulcorner}
A(\M_{12})\ar[r]^{A(i_1)}\ar[d]_{-A(i_2)} & A(\M_1)\ar[d]\\
A(\M_2)\ar[r] & P
}
\end{gathered}
\end{equation}
We can express the algebra $P$ variously as
\begin{align*}
P&\cong  \k\otimes_{A(\M_{12})}\big(A(\M_1)\otimes_\k A(\M_2)\big),\\
&\cong  \k\otimes_{A(\M_{12})}A(\M_1\oplus\M_2) &&
\text{by \cite[Prop.~13.1]{PS10};}\\
&\cong  A(\M_1\oplus\M_2)/\!\!/A(\M_{12}) &&
\text{see, e.g., \cite[p.~349]{CEbook};}\\
&\cong  A(\M_1\pc_X \M_2).
\end{align*}
\begin{rmk}
We note that, since $X$ is modular in $\M_1$,
$A(\M_1)$ is a free $A(\M_{12})$-module (see~\cite{Ter86}).  This implies
$A(\M_1\oplus\M_2)$ is also a free $A(\M_{12})$-module.  Taking Hilbert
series, we obtain
\begin{equation}\label{eq:pc hs}
h(A(\M_1\pc_X\M_2,t)=h(A(\M_1),t)h(A(\M_2),t)/h(A(\M_{12}),t).
\end{equation}
This amounts to the classical formula relating the characteristic polynomials
of the four matroids, so the diagram \eqref{eq:pc pushout} can be taken as
an algebraic refinement.
\end{rmk}
We now restrict to the classical parallel connection, where $X$ consists of 
a single element.  The description
above is particularly straightforward in this case: in particular, 
the linear map $\overline{\phi}$ in \eqref{eq:pc_Vs} is an isomorphism,
and we show next that it induces an algebra isomorphism.
In the case of complex hyperplane arrangements, the maps below
come from maps of spaces: see \cite{EF99} and \cite[\S 7]{DS12}.
If $X=\set{i}$ for some $i\in S_1\cup S_2$, let $i'$
denote its image in a copy of $S_2$ disjoint from $S_1$.  
\begin{theorem}\label{thm:pc}
If $X=\set{i}$, there is a short exact sequence
\[
\begin{tikzpicture}
\node (a) at (0,0) {$0$};
\node (b) at (2.5,0) {$A(\M_1\pc_{\set{i}}\M_2)[-1]$};
\node (c) at (5.8,0) {$A(\M_1\oplus\M_2)$};
\node (d) at (8.8,0) {$A(\M_1\pc_{\set{i}} \M_2)$};
\node (e) at (11,0) {$0$.};
\draw[->] (a) to (b);
\draw[->] (b) to (c);
\draw[->] (c) to node[above] {$\scriptstyle A(\phi)$} (d);
\draw[->] (d) to (e);
\end{tikzpicture}
\]
and an isomorphism 
$\oA(\M_1)\otimes_\k \oA(\M_2)\cong \oA(\M_1\pc_{\set{i}}\M_2)$.
\end{theorem}
\begin{proof}
Since $\phi$ is surjective, so is $A(\phi)$ (Theorem~\ref{thm:functorial}),
and $A(\M_1\pc_{\set{i}} \M_2)\cong A(\M_1\oplus\M_2)/(r)$, 
where $r=e_i-e_{i'}$.  
On the other hand, multiplication by $r$ gives a degree-$1$ map
$A(\M_1\oplus\M_2)/(r)\to A(\M_1\oplus\M_2)$.  Since $r$ is easily seen to be
nonresonant, this map is injective: that is, an isomorphism onto its
image, $(r)=\ker A(\phi)$.  It follows that the sequence is exact.

To prove the second claim, note the image of the restriction of $A(\phi)$ to 
$\oA(\M_1)\otimes_\k \oA(\M_2)$ is contained in $\oA(\M_1\pc_{\set{i}}\M_2)$,
since the source is generated in degree $1$, where the situation is 
that of \eqref{eq:pc_Vs}.  Since the target is also generated in degree $1$,
it follows the map is surjective.  To conclude it is
an isomorphism, we compare Hilbert series using
\eqref{eq:pc hs} and \eqref{eq:proj ses 1}.
\qed\end{proof}
The effect on resonance varieties is immediate.
\begin{corollary}\label{cor:pc}
If $X=\set{i}$, the map $\overline{\phi}\colon \oV(\M_1)\oplus
\oV(\M_2)\to \oV(\M_1\pc_X\M_2)$ restricts to an isomorphism for each $p\geq0$:
\[
\R^p(\M_1\oplus \M_2)\cong \R^p(\M_1\pc_X \M_2).
\]
\end{corollary}
\section{Singular subspaces and multinets}\label{sec:singular}
So far, we have seen that the resonance varieties of a matroid in 
top and bottom degrees are easy to account for, and that various
resonance components can be obtained from these by comparing with
submatroids, duals, deletion and contraction.  The lower bounds
obtained in this way sometimes match the upper bound given by
Theorem~\ref{thm:CDO bound}.

Some quite special matroids are known to 
have additional components in $\R^1(\M)$, however.  
We will assume from now on that $\chr\k=0$.
Building
on work of Libgober and Yuzvinsky~\cite{LY00} as well as Falk~\cite{Fa97}, 
Falk and Yuzvinsky~\cite{FY07} have characterized these in terms of
auxiliary combinatorics.  This is the notion of a multinet, and 
we we briefly recall the construction from \cite{FY07} in \S\ref{ss:multinets}
with a view to higher-degree generalizations.  We refer to Yuzvinsky's 
survey~\cite{Yu12} in particular for a complete introduction.

Some first steps generalizing this theory to $\R^p(\M)$ for $p>1$ appear
in \cite{CDFV12}, in the case of complex hyperplane arrangements, as 
well as in forthcoming work of Falk~\cite{BFW}.  We interpret these 
constructions in terms of maps of Orlik-Solomon algebras, as in 
\S\ref{sec:pieces}.

\subsection{$\R^1(\M)$: multinets}\label{ss:multinets}
\begin{definition}\label{def:multinet}
If $\M$ is a matroid on $[n]$ and $k$ is an integer with $k\geq3$,
a $(k,d)$-multinet is a partition $\L:=\set{\L_1,\ldots,\L_k}$ of $[n]$
together with a set $\X\subseteq L_2^{\irr}(\M)$ 
with the following four properties:
\begin{enumerate}
\item $\abs{\L_s}=d$ for all $1\leq s\leq k$;
\item If $i,j\in [n]$ belong to different parts of the partition $\L$,
then they span a flat in $\X$;
\item For any $X\in\X$, the number $\abs{\L_s\cap X}$ is 
independent of $s$;
\item For any $i,j\in\L_s$ for some $s$, there is a sequence
$i=i_0,i_1,\ldots,i_r=j$ for which $\spn\set{i_{q-1},i_q}\not\in \X$ for all
$1\leq q \leq r$.
\end{enumerate} 
\end{definition}
The original (equivalent) formulation in \cite{FY07} replaces $\M$ with its
simplification $\M_s$ and records the number of parallel elements
with a multiplicity function.  We will say that a simple matroid supports
a multinet if it is the simplification of a matroid with a partition as
in Definition~\ref{def:multinet}.

\begin{theorem}[Thms.~2.3, 2.4, \cite{FY07}]\label{thm:multinets}
$\R^1_\k(\M)$ contains an essential component if and only if 
$\M$ supports a multinet.
\end{theorem}
Explicitly, this component is the linear space $Q_{\L}\cap\oV$ 
(Definition~\ref{def:res subspaces}).
Their construction can be interpreted in terms of maps of Orlik-Solomon
algebras as follows.  First, we note that if $\L$ is a multinet, the 
partition satisfies
the hypothesis of Proposition~\ref{prop:partition}, so there is a morphism
$p_{\L}\colon \M\to \U_{2,k}$ and a surjection
\[
A(p_\L)\colon \oA(\M)\to \oA(\U_{2,k}).
\]

The more interesting aspect is the existence of a right inverse to 
$A(p_\L)$.  
Multinets give the following construction \cite{FY07}.
\begin{proposition}
If $(\L,\X)$ is a multinet on $\M$, there is a ring homomorphism
$i_{\L}\colon A(\U_{2,k})\to A(\M)$ defined by 
\[
i_{\L}(e_s)=\frac{1}{\abs{\L_s}}
\sum_{i\in\L_s}e_i\text{~for all $1\leq s\leq k$},
\]
which restricts to a map $i_{\L}\colon \oA(\U_{2,k})\to \oA(\M)$.
\end{proposition}
We note that $A(p_{\L})\circ i_{\L}=\id$.  From this it follows
that $i_{\L}$ is injective in cohomology.
From Example~\ref{ex:uniform}, 
$\R^1(\U_{2,k})=\oV(\U_{2,k})$, and its image in $\R^1(\M)$ is just
$Q_{\L}\cap\oV$.

\begin{exm}[\cite{FY07}]\label{ex:B3}
Let $\M$ be the matroid of the $B_3$ root system, giving each of the
short roots multiplicity two.  The corresponding 
hyperplane arrangement is defined
by the polynomial $x^2y^2z^2(x-y)(x+y)(x-z)(x+z)(y-z)(y+z)$.  Numbering
the points of the matroid $\set{1,1',2,2',3,3',4,5,6,7,8,9}$, the
dependencies are shown in Figure~\ref{fig:B3}.
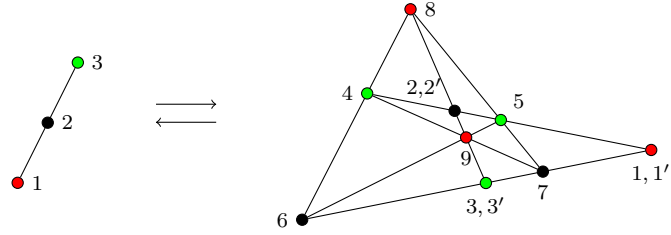
\begin{figure}
\[
\begin{tikzpicture}[baseline=(current bounding box.center),scale=0.8,
vertex/.style={circle,draw,inner sep=1.5pt,fill=black}]
\node[vertex,fill=red] (1) at (-0.5,-1) [label=right:$1$] {};
\node[vertex] (2) at (0,0) [label=right:$2$] {};
\node[vertex,fill=green] (3) at (0.5,1) [label=right:$3$] {};
\draw (1) -- (3);
\end{tikzpicture}
\qquad
\begin{tikzpicture}[scale=0.8]
\draw [->] (0,0) to (1,0);
\draw [->] (1,-0.3) to (0,-0.3);
\end{tikzpicture}
\qquad
\lower15pt\hbox{
\begin{tikzpicture}[baseline=(current bounding box.center),scale=0.8,
vertex/.style={circle,draw,inner sep=1.5pt,fill=black}]
\node[vertex,fill=red] (1) at (-0.2,3.5) [label=right:$8$] {};
\node[vertex] (4) at (-2,0) [label=left:$6$] {};
\node[vertex] (3) at (2,0.8) [label=below:$7$] {};
\node[vertex,fill=red] (a) at ($(4)!1.45!(3)$) [label=below:$1{,}\,1'$] {};
\draw[name path=t1] (1) -- (4);
\node[vertex,fill=green] (2) at ($(1)!0.4!(4)$) [label=left:$4$] {};
\draw[name path=t2] (1) -- (3);
\draw[name path=t3] (a) -- (2);
\path[name intersections={of=t3 and t2,by=5}];
\draw[name path=t4] (4) -- (a);
\draw[name path=t5] (2) -- (3);
\draw[name path=t6] (4) -- (5);
\path[name intersections={of=t5 and t6,by=6}];
\path[name path=t7b] (1) -- ($(1)!2.5!(6)$);
\path[name intersections={of=t3 and t7b,by=b}];
\node[vertex] at (b) [label=above left:$2{,}2'$] {};
\path[name intersections={of=t4 and t7b,by=c}];
\node[vertex,fill=green] at (5) [label=above right:$5$] {};
\draw (1) -- (c);
\node[vertex,fill=red] at (6) [label=below:$9$] {};
\node[vertex,fill=green] at (c) [label=below:$3{,}\,3'$] {};
\end{tikzpicture}
}
\]
\caption{The $(3,4)$-multinet for the $B_3$ root system}\label{fig:B3}
\end{figure}
The multinet $\L=\set{11'89|22'67|33'45}$ has $\X=L^{\irr}_2(\M)$,
so $Q_{\L}\cap\oV$ is a $2$-dimensional, essential component of $\R^1(\M)$.

Returning to the upper bound of Corollary~\ref{cor:ess bound},
let $\Co=L^{\irr}_2(B_3)$, an essential cover.  
By direct computation, $P_{\Co}$ has dimension
$2$, so it equals the essential component computed above.
\end{exm}

\subsection{$\R^{\geq1}(\M)$: singular subspaces}
The Multinet Theorem~\ref{thm:multinets} does not yet have 
a complete higher analogue.  Here, we indicate some first steps
in that direction, beginning with a definition from~\cite[\S3]{CDFV12}.  
\begin{definition}\label{def:singular}
A subspace $W\subseteq\oV(\M)$ is called {\em singular} if
the multiplication map $\Lambda^k(W)\to A^k(\M)$ is zero, where 
$k=\dim W$.  The {\em rank} of $W$ is the largest $q$ for which
$\Lambda^q(W)\to A^q(\M)$ is not the zero map.
\end{definition}
\begin{proposition}\label{prop:singular}
If $\phi\colon \oA(\U_{q+1,k+1})\to \oA(\M)$ is a graded homomorphism
which is injective in degree $1$, 
then the degree-$1$ part of $\im\phi$ is a singular subspace of
dimension $k$ and rank at most $q$.

Conversely, if $W$ is a singular subspace of rank $q$ in $\oV(\M)$, 
there exists
a map $\phi$ as above for which $W=(\im\phi)^1$.
\end{proposition}
\begin{proof}
Let $W\subseteq\oV$ be a subspace of dimension $k$.  By inspecting the
Orlik-Solomon relations
\eqref{eq:OSrelations}, we can identify $\oA(\U_{q+1,k+1})$ with a truncated
exterior algebra $\Lambda(W)/(\Lambda^{q+1}(W))$.

If a map $\phi\colon \oA(\U_{q+1,k+1})\to\oA(\M)$ is given, let 
$W=\phi(\oV(\U_{q+1,k+1}))$.
Since $\oA(\U_{q+1,k})^p=0$ for $p>q$, its image $W$ 
is singular of rank at most $q$.

Conversely, the hypothesis implies that the natural map 
$\Lambda(W)\to A(\M)$ factors through 
$\oA(\U_{q+1,k+1})\cong\Lambda(W)/(\Lambda^{q+1}(W))$.
\qed\end{proof}

If $\chr\k=0$, then 
components of $\R^1(\M)$ are just the same as rank-$1$ singular subspaces,
by Theorem~\ref{thm:isotropic}.  
For higher rank, the situation is more subtle.  If $W$ is a singular subspace
of rank $q$ and dimension $k$, the condition
implies that the natural homomorphism $\Lambda(W)\to
\oA(\M)$ factors through $\oA(U_{q+1,k+1})$, a truncated exterior algebra.
If the resulting 
map $\oA(U_{q+1,k+1})\to\oA(\M)$ is injective in cohomology, then
$W\subseteq \R^q(\M)$.

\begin{exm}\label{ex:pyramid}
The graph $G$ from Example~\ref{ex:pyramid1} provides an interesting example
of a singular subspace.
$\R^1(G)$ consists of the four local components from the three-element
flats.  $\R^2(G)$ is more complicated.  We find two essential components
by first constructing a singular subspace.
We label the vertices of $G$ with $\set{1,2,\ldots,5}$
so that edge $i$ has vertices $(i,i+1)$ for $i=1,2,3$.  If edge $i=\set{s,t}$
and $s<t$, let $f_i=x_t-x_s\in\k[x_1,\ldots,x_5]$.  The linear forms 
$\set{f_i\colon 1\leq i\leq 8}$ define the graphic
arrangement $\A(G)$ with matroid $\M(G)$.

Following the approach of Cohen et al.\ \cite{CDFV12}, 
we observe that there is a 
linear relation of cubic polynomials
\[
f_1f_7f_8+f_2f_5f_8+f_3f_5f_6=f_4f_6f_7.
\]
This implies that the polynomial mapping $\Phi\colon U(\A(G))\subseteq\P^4\to
\P^2$ given by
\[
\Phi(x)= [f_1f_7f_8:f_2f_5f_8:f_3f_5f_6]
\]
has its image inside the projective complement of the arrangement consisting of 
the three coordinate hyperplanes together with the projective 
hyperplane orthogonal to $[1:1:1]$.  The existence of the 
induced map in cohomology $\Phi^*\colon\oA(\U_{3,4})\to
\oA(G)$ shows that $W$ is a rank-$2$ singular subspace, by
Proposition~\ref{prop:singular}, where
\[
W=\set{v\in\k^8\colon v=(a,b,c,d,b+c,c+d,d+a,a+b)\text{~and~}a+b+c+d=0}.
\]
Below, we will show $\Phi^*$ is split, which implies that
$H^2(\Phi^*)$ is injective, and $W\subseteq\R^2(G)$.  

We proceed indirectly to show that $W$ is maximal (i.e., a component.)
First, using the cover 
$\Co=L^{\irr}_2(\M)=\set{156,267,378,458}$, we obtain a $3$-dimensional
linear space
\[
P_{\Co}=\set{v\in\k^8\colon v=(c+d,d+a,a+b,b+c,a,b,c,d)\text{~and~}a+b+c+d=0}.
\]
Note $P_{\Co}$ is not maximal in $\Bound^2(G)$: for example,
if $\Co'=\set{1234,156,378}$, $P_{\Co}\subsetneq P_{\Co'}$.  Up to symmetry,
though, this is the only subspace in $\Bound^2(G)$ that properly contains
$P_{\Co}$, and it has dimension $4$.  By checking a single 
$v\in P_{\Co'}-P_{\Co}$, we see $P_{\Co'}\not\subseteq\R^2(G)$.  

Now we note that the matroid $\M(G)$ is self-dual, and 
we may identify $\M(G)$ with its dual via the permutation $\sigma=[56784123]$.
Since $W$ is essential, $\sigma(W)\subset\R^2(G)$ 
as well, by Theorem~\ref{thm:res_basics}(3).  However, $\sigma(W)=P_{\Co}$,
so $P_{\Co}\subseteq\R^2(G)$ as well.  It follows that
$W$ and $\sigma(W)$ are both (maximal) linear components of $\R^2(G)$.

In order to try to imitate the multinet construction
(Theorem~\ref{thm:multinets}), we give the inner four edges of $G$ 
multiplicity $2$, and denote the (non-simple) matroid by $\widetilde{\M}$.  Let
$\L$ be the partition of $\set{1,2,3,4,5,5',6,6',7,7',8,8'}$ given by
\begin{equation*}
\L\,=\,
\begin{tikzpicture}[anchor=base,baseline=10pt,scale=0.35,
vertex/.style={circle,draw,inner sep=1pt,fill=black}]

\node[vertex] (a) at (0,0) {};
\node[vertex] (b) at (2,0) {};
\node[vertex] (c) at (2,2) {};
\node[vertex] (d) at (0,2) {};
\node[vertex] (e) at (1,1) {};

\draw (a) to node[below] {1} (b);
\draw (c) to node[right] {7} (e);
\draw (d) to node[left] {8} (e);
\end{tikzpicture}
\quad\big|\quad
\begin{tikzpicture}[anchor=base,baseline=10pt,scale=0.35,
vertex/.style={circle,draw,inner sep=1pt,fill=black}]

\node[vertex] (a) at (0,0) {};
\node[vertex] (b) at (2,0) {};
\node[vertex] (c) at (2,2) {};
\node[vertex] (d) at (0,2) {};
\node[vertex] (e) at (1,1) {};

\draw (b) to node[right] {2} (c);
\draw (a) to node[below] {5} (e);
\draw (d) to node[above] {$8'$} (e);
\end{tikzpicture}
\,\big|\quad
\begin{tikzpicture}[anchor=base,baseline=10pt,scale=0.35,
vertex/.style={circle,draw,inner sep=1pt,fill=black}]

\node[vertex] (a) at (0,0) {};
\node[vertex] (b) at (2,0) {};
\node[vertex] (c) at (2,2) {};
\node[vertex] (d) at (0,2) {};
\node[vertex] (e) at (1,1) {};

\draw (c) to node[above] {3} (d);
\draw (a) to node[left] {$5'$} (e);
\draw (b) to node[right] {6} (e);
\end{tikzpicture}
\quad\big|\,
\begin{tikzpicture}[anchor=base,baseline=10pt,scale=0.35,
vertex/.style={circle,draw,inner sep=1pt,fill=black}]

\node[vertex] (a) at (0,0) {};
\node[vertex] (b) at (2,0) {};
\node[vertex] (c) at (2,2) {};
\node[vertex] (d) at (0,2) {};
\node[vertex] (e) at (1,1) {};

\draw (d) to node[left] {4} (a);
\draw (b) to node[below] {$6'$} (e);
\draw (c) to node[above] {$7'$} (e);
\end{tikzpicture}.
\end{equation*}
Then $W\cong Q_\L\cap \oV(\widetilde{\M})$.

The map $A(\U_{3,4})\to A(\widetilde{\M})$
given by sending the $i$th generator to the sum of elements in
the $i$th block of $\L$, for 
$1\leq i\leq4$, restricts to $\Phi^*$ above.
To construct a left-inverse, recall that in Example~\ref{ex:weak}
we found a morphism $f\colon\M\to \U_{3,4}$ in $\oMM$.
A simple check shows that $-A(f)\circ\Phi^*$ is the 
identity on $\oA(\U_{34})$, so $\Phi^*$ is indeed split.
\end{exm}

We continue this example by observing that not every essential component
of $\R^2(G)$ is a singular subspace.
\begin{exm}[Example~\ref{ex:pyramid}, continued]\label{ex:pyramid2}
The second component $P_{\Co}=\sigma(W)$ can also be expressed in terms of 
a partition: $\sigma(W)\cong Q_{\L^*}\cap \oV(\widetilde{\M})$ for another
non-simple matroid $\widetilde{M}$, where
\begin{equation*}
\L^*\,=\,
\begin{tikzpicture}[anchor=base,baseline=10pt,scale=0.35,
vertex/.style={circle,draw,inner sep=1pt,fill=black}]

\node[vertex] (a) at (0,0) {};
\node[vertex] (b) at (2,0) {};
\node[vertex] (c) at (2,2) {};
\node[vertex] (d) at (0,2) {};
\node[vertex] (e) at (1,1) {};

\draw (b) to node[right] {2} (c);
\draw (c) to node[above] {3} (d);

\draw (a) to node[below] {5} (e);
\end{tikzpicture}
\,\big|\,
\begin{tikzpicture}[anchor=base,baseline=10pt,scale=0.35,
vertex/.style={circle,draw,inner sep=1pt,fill=black}]

\node[vertex] (a) at (0,0) {};
\node[vertex] (b) at (2,0) {};
\node[vertex] (c) at (2,2) {};
\node[vertex] (d) at (0,2) {};
\node[vertex] (e) at (1,1) {};

\draw (c) to node[above] {$3'$} (d);
\draw (d) to node[left] {4} (a);
\draw (b) to node[below] {6} (e);
\end{tikzpicture}
\quad\big|\,
\begin{tikzpicture}[anchor=base,baseline=10pt,scale=0.35,
vertex/.style={circle,draw,inner sep=1pt,fill=black}]

\node[vertex] (a) at (0,0) {};
\node[vertex] (b) at (2,0) {};
\node[vertex] (c) at (2,2) {};
\node[vertex] (d) at (0,2) {};
\node[vertex] (e) at (1,1) {};

\draw (a) to node[below] {1} (b);
\draw (d) to node[left] {$4'$} (a);
\draw (c) to node[above] {7} (e);
\end{tikzpicture}
\quad\big|\quad
\begin{tikzpicture}[anchor=base,baseline=10pt,scale=0.35,
vertex/.style={circle,draw,inner sep=1pt,fill=black}]

\node[vertex] (a) at (0,0) {};
\node[vertex] (b) at (2,0) {};
\node[vertex] (c) at (2,2) {};
\node[vertex] (d) at (0,2) {};
\node[vertex] (e) at (1,1) {};

\draw (a) to node[below] {$1'$} (b);
\draw (b) to node[right] {$2'$} (c);
\draw (d) to node[above] {8} (e);
\end{tikzpicture}.
\end{equation*}
The image of $\Lambda^3(\sigma(W))$ in $\oA(G)$ is nonzero,
so this subalgebra generated by $\sigma(W)$ does not factor through
$\oA(\U_{3,4})$.
\end{exm}

However, the partitions $\L$
and $\L^*$ above look qualitatively rather similar.  One might hope,
then, that Theorem~\ref{thm:multinets} admits a combinatorial
generalization that treats both components above equally.

\begin{ques}
By the theory of multinets, every component of $\R^1(\M)$ comes
from the tautological resonance of 
a rank-$2$ matroid, via a split surjection of 
Orlik-Solomon algebras.   Does every component of $\R^p(\M)$
come from a matroid of rank $p+1$, for all $p\geq1$?  

Our last example shows that we cannot always find a uniform 
matroid with this property for $p=2$, unlike for $p=1$; however,
since the only simple matroids of rank $2$ are uniform, this
should not necessarily be seen evidence that the answer is negative.
\end{ques}

\begin{ack}
The author would like to thank Hal Schenck for the ongoing conversations
from which the main ideas for this paper emerged.
\end{ack}

\providecommand{\bysame}{\leavevmode\hbox to3em{\hrulefill}\thinspace}
\providecommand{\MR}{\relax\ifhmode\unskip\space\fi MR }
\providecommand{\MRhref}[2]{%
  \href{http://www.ams.org/mathscinet-getitem?mr=#1}{#2}
}
\providecommand{\href}[2]{#2}

\end{document}